\RequirePackage{fix-cm}
\documentclass[smallextended]{svjour3}       %
\smartqed

\usepackage{graphicx}
\usepackage{amsmath}
\usepackage{amssymb}
\usepackage{color}
\usepackage{verbatim}
\usepackage{tikz}
\usepackage{subfig}
\usepackage{url}
\usepackage[multiple]{footmisc}

\usetikzlibrary{shapes.multipart}
\usepackage[utf8]{inputenc}

\DeclareMathOperator{\rank}{rank}

\title{A continuation approach to the optimization of hydropower operations}
\author{Jorn H. Baayen \and Julia Rauw \and Teresa Piovesan}

\institute{Jorn H. Baayen \at KISTERS Nederland B.V.,
              Piet Mondriaanplein 13-31, 3812 GZ Amersfoort, The Netherlands \\
              \email{jorn.baayen@kisters-bv.nl} \\
              Telephone: +49 (0)151 54433908 \\
              ORCID 0000-0002-1905-3677 \\
            \and 
            Julia Rauw \at  FH Aachen - University of Applied Sciences, Faculty of Mechanical 				Engineering and Mechatronics,
			Goethestrasse 1, 52064 Aachen, Germany\\
           \and
           Teresa Piovesan \at
            Deltares,
              Postbus 177,
              2600 MH Delft,
              The Netherlands
}

\date{\today}

\begin{document}
\maketitle

\begin{abstract}
The instantaneous power generation from a hydroelectric turbine is proportional to the product of head difference and turbine flow.  The equation relating power to hydraulic variables is therefore nonlinear.  Hence, optimization problems subject to this relation, such as release schedule optimization, are nonconvex and may admit multiple local isolated minima.  This renders such problems problematic for use in operational model predictive control.  This paper shows that release schedule optimization problems subject to the nonlinear turbine generation equation may be set up using a continuation approach to be both \emph{zero-convex} and \emph{path stable}.  In this way such optimization problems become suitable for decision support systems based on model predictive control.  An example problem is studied, and it is shown that significant productivity gains may be realized using the presented methodology.
\end{abstract}

\section{Introduction}

The question of how to best operate a cascade of reservoirs for hydroelectric power generation has been studied for more than half a century \cite{lane1944tva,yeh1985reservoir,labadie2004optimal,reed2013evolutionary}, and probably longer.  It is a rich question in the sense that it involves several challenging factors:  (I) multiple goals and objectives, (II) inflow uncertainty, (III) quantification of the value of reservoir water at the end of the planning horizon, and (IV) the nonlinear relation between discharge, head difference, and hydroelectric power generation.

With regards to (I), the \emph{multiple objectives} typically stem from the various competing uses of reservoir water, such as water supply, flood control, recreation, and power generation.  Typically, a priority ordering is assigned to the various uses, which then maps naturally to a multi-objective optimization technique knows as lexicographic or sequential goal programming, see, e.g., \cite{Eschenbach2001,gijsbers2017}.

With regards to (II), \emph{inflow uncertainty}, much has been published, see e.g. \cite{pereira1989optimal,raso2014short,Gauvin2015,Braaten2016}.  Various satisfying techniques are available to incorporate this uncertainty.  Further discussion of these is beyond the scope of this work.

With regards to (III), the \emph{terminal cost problem}, the solution for short term planning is to refer to the annual rule curve;  for long term planning, annual periodicity may be imposed.  Alternatively, techniques based on the approximation of the Bellman value function or using Pontryagin's maximum principle produce optimal releases based on current system state only and hence do not require a terminal cost to be specified \cite{pereira1989optimal}, at the expense of being unable to react to salient features in short-term forecasts.  Again, further discussion of these issues is beyond the scope of this work.

In this paper, we focus on (IV), the \emph{generation nonlinearity}, 
\begin{equation}\label{eq:core}
P = g \rho \eta(Q,H_u,H_d) Q \Delta H,
\end{equation}
with instantaneous hydroelectric generation $P$, gravitational constant $g$, density of water $\rho$, reservoir water level $H_u$, downstream water level $H_d$, turbine efficiency $\eta=\eta(Q,H_u,H_d)$, and head difference $\Delta H := H_u - H_d$.  This equation is nonlinear and hence, if included as an equality constraint in an optimization problem, results in a formulation that is neither linear nor convex.  As optimization problems which are not convex admit multiple locally optimal solutions \cite{Boyd2004}, this results in a state of affairs where different optimization algorithms, whether gradient-based or evolutionary, may result in different solutions.  Worse, different settings of the same algorithm may produce different results.  This state of affairs is highly undesirable in an operational setting, where predictable and consistent decision support is expected.

Further insight is obtained by noting that, to obtain a particular generation $P$, this can either be achieved by a ``high'' turbine flow at ``low'' head, or by ``low'' turbine flow at ``high'' head; see also the isolines in Figure \ref{fig:isolines}.  Without a preference for low or high head being articulated, the generation request tracking problem is ambiguous.

\begin{figure}[h]
\centering
  \includegraphics[scale=0.5]{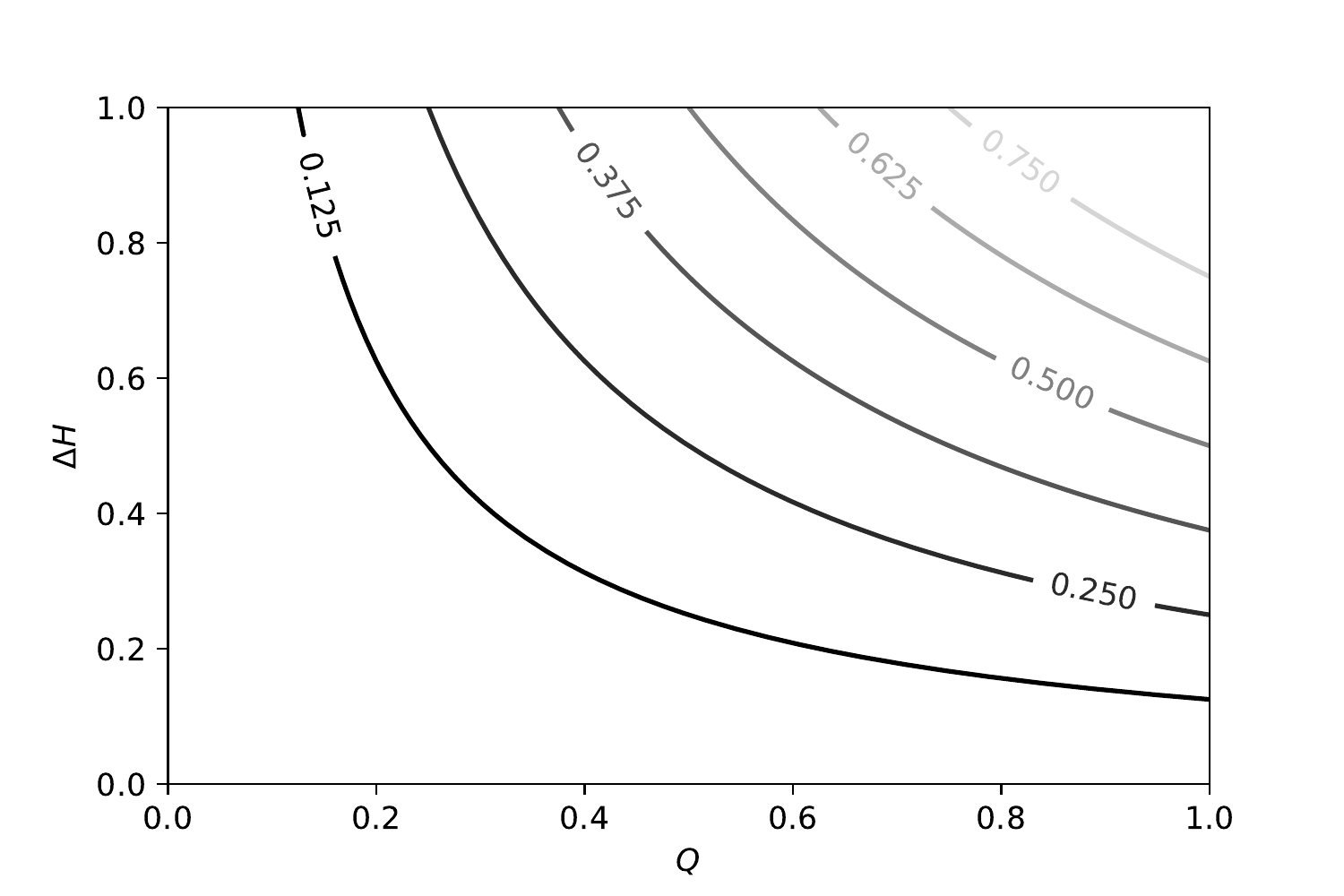}
\caption{Isolines of the function $f=Q\Delta H$ for $Q\in[0,1]$ and $\Delta H \in[0,1]$.}
\label{fig:isolines}       
\end{figure}

Typically, the nonlinearity is removed either by linearizing Equation \ref{eq:core}, for instance by fixing the efficiency and the head difference, i.e., by using $Q(\Delta H)_0$ for some fixed $(\Delta H)_0$ as opposed to the nonlinear product in Equation \ref{eq:core}, or by piecewise linearization, whereby an integer variable is used to select the active linear model \cite{labadie2004optimal}.  Standard linearization is inaccurate for planning problems in which reservoir water levels vary significantly, whereas piecewise linearization introduces integer decision variables.  Integer variables introduce the \emph{curse of dimensionality}, i.e., rapid explosion of computational complexity as new reservoirs and/or time steps are introduced. Other authors include the full nonlinear relation in the optimization model and settle for a locally optimal solution close to a seed solution \cite{Schwanenberg2015}.

This work presents a new approach towards Equation \ref{eq:core}. Based on earlier work on dynamic optimization subject to the shallow water equations \cite{Baayen2018}, we show how a continuation method can be used to solve optimization problems subject to the nonlinear Equation \ref{eq:core}, while retaining control over and insight into the local minimum obtained. 
In a nutshell the continuation method, also known as homotopy, is the process of continuously deforming of a mathematical object into another. The idea is to smoothly transition from a linear optimization problem to a nonlinear one while tracking the solutions obtained throughout this deformation. 
The approach is readily integrated with methods for challenges (I)-(III) such as goal programming and multi-stage stochastic optimization.

\section{Parametric programming with continuation}

In our earlier publication \cite{Baayen2018}, we consider parametric optimization problems of the form
\begin{align}
 \min_x f(x,\theta) \quad & \nonumber\\
 \text{subject to} \quad & c(x,\theta) = 0, \tag{$\mathcal P^\theta$}\label{def:parametric-optimization-problem}\\
& x \geq 0 \nonumber;
\end{align}
where $x$ is the optimization variable, $\theta \in [0,1]$ is a parameter, $f(x,\theta)$ is the objective function and $c(x,\theta) : \mathbb R^{n} \times [0,1] \to \mathbb R^{\ell}$ denotes the set of constraints. 
Throughout we assume that the functions $f$ and $c$ are at least twice differentiable.
This class of optimization problems can be solved using the so-called interior point methods type of algorithms. 
For these one has to solve a sequence of barrier problems
\begin{gather}
\min_x f(x, \theta) - \mu \sum_{i=1}^n \ln (x_i)  \nonumber\\
\text{subject to} 
\quad c(x,\theta) = 0, \tag{$\mathcal P_\mu^\theta$}\label{def:barrier-optimization-problem}\\
\hspace*{15mm}
x \in \mathbb{R}^n; \nonumber
\end{gather}
for a non-negative decreasing sequence of barrier parameters $\mu$ converging to zero. The key idea is that the optimum value of $(\mathcal P_\mu^\theta)$ converges to an optimal solution of $(\mathcal P^\theta)$ as $\mu$ tends to zero.

Consider the Lagrangian function $\mathcal L_\mu (x, \lambda, \theta)$ for the barrier problem (\ref{def:barrier-optimization-problem}), that is 
 $$\mathcal L_\mu (x, \lambda, \theta) := f(x, \theta) - \mu \sum_{i=1}^n \ln (x_i) + \lambda^T c(x,\theta),$$ where $\lambda$ is the vector of Lagrangian multipliers.
Then any solution of (\ref{def:barrier-optimization-problem}) is a solution of the  system of equations:
\begin{eqnarray}\label{eq:primal}
\nabla_{x} \mathcal L_\mu (x, \lambda, \theta)  & = & 0, \label{eq:primal1}\\
c(x, \theta) & = & 0. \nonumber \label{eq:primal2}
\end{eqnarray}

Let $F_\mu(x,\lambda,\theta)$ be the residual of the system of equations (\ref{eq:primal}). 
The system
\[
F_\mu(x,\lambda,\theta)=0
\]
admits a unique solution path in the neighborhood of $x^*$, $\lambda^*$, and $\theta^*$ if the Jacobian $\partial F_\mu(x^*,\lambda^*,\theta^*) / \partial (x,\lambda)$ is nonsingular. 

\begin{definition}\label{def:critical-point}
For some fixed parameters $\hat \theta \in [0,1]$ and $\mu > 0$,
let $\hat x$ be a solution of the parametric barrier problem (\ref{def:barrier-optimization-problem}) and $\hat \lambda$ be the corresponding Lagrange multipliers.  The point $(\hat x,\hat \lambda,\hat \theta)$ is called a \emph{critical point} if the Jacobian $\partial F_\mu(\hat x,\hat \lambda , \hat \theta)/\partial (x, \lambda)$ is singular.
\end{definition}

To characterize the circumstances under which the parametric deformation process works well, i.e., does not admit singular points, we refer to the following two notions introduced in \cite{Baayen2018}:

\begin{definition}\label{def:zero-convex}
We say that the parametric optimization problem (\ref{def:parametric-optimization-problem}) is \emph{zero-convex} if $(\mathcal P^0)$ is a convex optimization problem. 
That is, if $x \mapsto f(x,0)$ is a convex function and if $x \mapsto c(x,0)$ is linear.
\end{definition}

The notion of zero-convexity captures the idea that there should be a unique solution at $\theta=0$, and that such solution can be found using standard methods.  

\begin{definition}\label{def:path-stable}
We say that the parametric optimization problem (\ref{def:parametric-optimization-problem}) is \emph{path stable with respect to the interior point method} if its barrier formulation (\ref{def:barrier-optimization-problem}) does not admit critical points for any $\theta \in [0,1]$ and any $\mu > 0$.
\end{definition}

Path stability implies that any solution path does not bifurcate or terminate and hence it is possible to track a local minimum of the optimization problem ($\mathcal P^\theta$) as we continuously modify the parameter $\theta$. 

Zero-convex and path stable problems can be solved with a simple continuation algorithm as described in \cite{Baayen2018}.
The same approach can be easily extended to parametric optimization problems with inequalities constraints and with bounded variables (see e.g.~\cite{Baayen2018}). 

In the following, we derive a parametric optimization problem that includes the nonlinear generation Equation \ref{eq:core} and that is zero-convex and path stable.

\section{Hydroelectric generation}

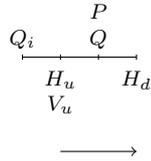
\begin{figure}[ht]
\centering
\begin{tikzpicture}
\draw (0.0,0) -- (1.5,0);
\draw (1 cm - 1 cm,1 pt) -- (1 cm - 1 cm,-1 pt) node[anchor=south] { \begin{tabular}{c} $Q_i$ \end{tabular} };
\draw (2 cm - 1 cm,1 pt) -- (2 cm - 1 cm,-1 pt) node[anchor=south] { \begin{tabular}{c} $P$ \\ $Q$ \end{tabular} };
\draw (1 cm - 0.5 cm,1 pt) -- (1 cm - 0.5 cm,-1 pt) node[anchor=north] {\begin{tabular}{c} $H_u$ \\ $V_u$ \end{tabular} };
\draw (2 cm - 0.5 cm,1 pt) -- (2 cm - 0.5 cm,-1 pt) node[anchor=north] {\begin{tabular}{c} $H_d$ \\ $\quad$ \end{tabular} };
\draw[->] (0.5,-1.25 cm) -- (1.5 cm,-1.25 cm) node[anchor=west] {};
\end{tikzpicture}
\caption{Hydraulic variables governing hydroelectric generation $P$.}
\label{fig:example-grid}
\end{figure}

The instantaneous generation of a hydroelectric turbine is given by Equation (\ref{eq:core}), which we repeat here:
\[
P = g \rho \eta(Q,H_u,H_d) Q \Delta H.
\]
The head difference $\Delta H = H_u - H_d$ is a function of the upstream water level $H_u$ and the downstream water level $H_d$.  These, in turn, typically depend on the volume of water contained in the upstream and downstream sections.  Consider, for instance, the upstream reservoir with volume $V_u$:
\begin{equation}\label{eq:level-volume}
H_u = \Gamma(V_u)
\end{equation}
with a strictly increasing, twice differentiable volume-level relation $\Gamma$. 
As the inverse function $\Gamma^{-1}$ clearly exists, the variable $V_u$ is fully determined by $H_u$. 
Moreover, as a reservoir's surface area generally increases with the water level, $\Gamma$ is typically a concave function.
The volume of water contained in the reservoir, $V_u$, is also influenced by the reservoir outflow $Q$ through the linear mass balance relation:
\begin{equation}\label{eq:mass-balance}
\frac{\partial V_u}{\partial t} = Q_i - Q,
\end{equation}
with upstream inflow $Q_i$.  The upstream inflow could either be a time series boundary condition, or, if it is the outflow from a reservoir further upstream, another optimization variable.  Similarly, $H_d$ could be a downstream boundary condition, or it could be related to the volume of a reservoir further downstream.  These hydraulic variables are shown in Figure 
(\ref{fig:example-grid}).

\medskip
We discretize the equations in time for equidistant timesteps $t_j$ for $j \in \{1, \dots, T\}$. Such discretization is straight forward for Equations~(\ref{eq:core})-(\ref{eq:level-volume}) and implicit for Equation~(\ref{eq:mass-balance}):
\begin{equation}\label{eq:discr:mass-bal}\tag{$c_{3,j}$}
\frac{V_u(t_j) - V_u(t_{j-1})}{\Delta t} - Q_i(t_j) + Q(t_j) = 0.
\end{equation}

\medskip
For technical reasons that will be apparent in the next section, we split the optimization variables into two types: the \emph{proper} and the \emph{redundant} ones. The main difference is that the constraints must be linear with respect with the proper variables, while the redundant variables can have nonlinear relationships.

Let us introduce an alias variable $\widetilde Q$, where $\widetilde Q = Q$.
We define the redundant variables $x_r$ to be $(\widetilde Q, V_u, \Delta H)$.
Note that we can reformulate the turbine efficiency function $\eta(Q,H_u, H_d)$ using the variables $\widetilde Q, V_u$ and $\Delta H$. Indeed, the following simple relationships hold: $Q = \widetilde Q$, $H_u = \Gamma^{-1}(V_u)$ and $H_d =  \Gamma^{-1}(V_u) - \Delta H$. 
Let $\widetilde \eta(\widetilde Q, V_u, \Delta H)$ be the appropriate reformulation of the turbine efficiency function, then we can rewrite Equation~(\ref{eq:core}) as 
\[
P = g \rho \widetilde \eta(\widetilde Q, V_u, \Delta H) \widetilde Q \Delta H.
\]

\medskip
The last ingredient we need in order to be able to apply the continuation method is the linear approximations of the nonlinear equations (\ref{eq:core}) and (\ref{eq:level-volume}). A linear model approximating Equation (\ref{eq:core}) is given by
\[
P(t_j) = g \rho \eta_0 Q(t_j)(\Delta H)_0 
\]
with scalars $\eta_0$ and $(\Delta H)_0$.  A linear level-volume relation (\ref{eq:level-volume}) is given as
\[
H_u(t_j)= \alpha_\Gamma V_u(t_j) + \beta_\Gamma
\]
for some appropriately chosen coefficients $\alpha_\Gamma$ and $\beta_\Gamma$.
We can assume $\alpha_\Gamma$ to be a positive scalar since the function $\Gamma$ is strictly increasing, that is, $\partial \Gamma / \partial V_u$ is strictly positive. 

\medskip
Using the homotopy parameter $\theta \in [0,1]$, we obtain the following parametric constraints:
\begin{align}\label{eq:homotopic-equation-1}\tag{$c_{1,j}$}
P(t_j) - g \rho \big[ & (1-\theta) \eta_0 Q(t_j) (\Delta H)_0 \\
&+ \theta \widetilde \eta(\widetilde Q(t_j), V_u(t_j), \Delta H(t_j)) \widetilde Q(t_j) \Delta H(t_j) \big] = 0 \nonumber
\end{align}
and
\begin{equation}\label{eq:homotopic-equation-2}\tag{$c_{2,j}$}
 H_u(t_j) - (1-\theta) (\alpha_\Gamma V_u(t_j) + \beta_\Gamma) - \theta \Gamma(V_u(t_j)) = 0.
\end{equation}

Beside the equations (\ref{eq:discr:mass-bal}), (\ref{eq:homotopic-equation-1}),  (\ref{eq:homotopic-equation-2}), our system will be constrained by the relations 
\begin{equation}\label{eq:head-diff}\tag{$c_{4,j}$}
\Delta H(t_j) - H_u(t_j) + H_d(t_j) = 0
\end{equation}
and
\begin{equation}\label{eq:alias}\tag{$c_{5,j}$}
Q(t_j) - \tilde{Q}(t_j) = 0,
\end{equation}
for any $j \in \{1, \dots, N\}$.

\section{Analysis of zero-convexity and path stability}

In this section we will show that zero-convexity and path-stability hold for optimization problems containing the constraints (\ref{eq:homotopic-equation-1}) -- (\ref{eq:alias}).  Throughout this section we call the optimization variables $(\widetilde Q, V_u, \Delta H)$ \emph{redundant} and all the other ones \emph{proper}. 

We need the following assumptions:

\begin{itemize}
\item [BND] All the proper optimization variables have a lower bound, an upper bound, or both. None of the redundant variables have bounds.
\item [OBJ] The objective function is of the form 
\[
f(x,\theta) = f_{r}(x_{r},\theta)+ \sum_{k \in K} a_k x_k^2 +\sum_{l \in L} b_l x_l 
\]
with $a_k,b_l \in\mathbb R$ such that $a_k \geq 0$ and $f_{r}(x_{r},\theta)$ is a function of the redundant variables $x_{r}$ such that $x \mapsto f_{r}(x_{r},0)$ is convex.
\item [LIN] All additional constraints are affine functions of the optimization variables.
\item [IND] The gradients of all equality constraints form a linearly independent set.
 \end{itemize}

\begin{lemma}\label{lem:lin}
The gradients for the constraints (\ref{eq:homotopic-equation-1}) -- (\ref{eq:alias}) are linearly independent. Moreover, the constraints (\ref{eq:homotopic-equation-2}), (\ref{eq:head-diff}) and (\ref{eq:alias}) span the space of the redundant variables $x_r$.
\end{lemma}

\begin{proof}
We first prove the linear independence. Consider a fix timestep $t_j$ and let $\hat x$ be the vector of the variables $P(t_j), H_u(t_j), H_d(t_j)), Q(t_j), Q_i(t_j)$

\begin{eqnarray*}
\left. \frac{\partial c_{1,j}}{\partial x}\right|_{\hat{x}} &= &
\big( 1, 0, 0, 0, 0 \big) , \\ 
\left. \frac{\partial c_{2,j}}{\partial x}\right|_{\hat{x}} &= & 
\big( 0, 1, 0, 0, 0 \big) , \\ 
\left. \frac{\partial c_{3,j}}{\partial x}\right|_{\hat{x}} &= &
\big( 0, 0, 0, 1, -1\big) , \\ 
\left. \frac{\partial c_{4,j}}{\partial x}\right|_{\hat{x}} &= & 
\big( 0, -1, 1 , 0, 0 \big) , \\ 
\left. \frac{\partial c_{5,j}}{\partial x}\right|_{\hat{x}} &= &
\big( 0, 0, 0, 1, 0 \big), \\ 
\end{eqnarray*}
It is clear that the set of vectors $\{ \partial c_{i,j} / \partial x : i \in \{1, \dots, 5\} \}$ is linearly independent for any $j \in \{1, \dots, T\}$.
Moreover, for any $j \in \{1,\dots,T\}$, since $c_{3,j}$ is the only constraint containing the variable $Q_i(t_j)$ and $\partial c_{i,j} / \partial x \mid_{Q_i(t_j)} = 1$, the gradient of $c_{3,j}$ is linearly independent from set of all the gradients of the other constraints. 
We can conclude the linear independence by noticing that for $i \in \{1,2,4,5\}$ the gradient $\partial c_{i,j} / \partial x$
has non-zero entries only for variables depending on the timestep $t_j$.

We now prove that the constraints (\ref{eq:homotopic-equation-2}), (\ref{eq:head-diff}) and (\ref{eq:alias}) span the space of redundant variables.
Order the redundant variables as $x_r  = (V_u(t_j), \Delta H(t_j), \tilde{Q(t_j)})$, then:
\begin{eqnarray*}
\left. \frac{\partial c_{2,j}}{\partial x}\right|_{x_r} &= & 
\big( \gamma, 0, 0 \big) , \\ 
\left. \frac{\partial c_{4,j}}{\partial x}\right|_{x_r} &= & 
\big( 0, 1, 0  \big) , \\ 
\left. \frac{\partial c_{5,j}}{\partial x}\right|_{x_r} &= &
\big( 0, 0, 1 \big), \\ 
\end{eqnarray*}
where $\gamma = - (1- \theta)\alpha_\Gamma - \theta \partial \Gamma(V_u(t_j)) / \partial V_u(t_j)$. This is strictly different from zero as both $\alpha_\Gamma$ and $\partial \Gamma(V_u(t_j)) / \partial V_u(t_j)$ are positive numbers.
Therefore, (\ref{eq:homotopic-equation-2}), (\ref{eq:head-diff}) and (\ref{eq:alias}) span the space of the redundant variables.
\end{proof}

\begin{lemma}\label{lemma:hessian2}
Assume BND, OBJ, and LIN. 
Let us partition the variables $x$ into $(x_p, x_r)$. 
Then the Hessian of the Lagrangian with respect to the variables $x$, $\nabla_{xx}^2 \mathcal{L}_{\mu}(x,\lambda, \theta)$,
has the following block structure:
\[
\nabla_{xx}^2 \mathcal{L}_{\mu}(x,\lambda, \theta) = \begin{pmatrix} \nabla_{x_{p}x_{p}}^2 \mathcal{L}_{\mu} & 0 \\
0 & \nabla_{x_{r}x_{r}}^2 \mathcal{L}_{\mu}
\end{pmatrix}
\]
where $\nabla_{x_{p}x_{p}}^2 \mathcal{L}_{\mu}$ is a diagonal nonsingular matrix.
\end{lemma}
\begin{proof}
Due to OBJ and LIN, the only non-zero off-diagonal entries of the Hessian are the ones corresponding to the redundant constraints $\nabla_{x_{r}x_{r}}^2 \mathcal{L}_{\mu}$.

Assumption BND implies that the second derivatives of the logarithmic barrier function give a positive contribution to all the diagonal entries of the matrix $\nabla_{x_{p}x_{p}}^2 \mathcal{L}_{\mu}$.
Moreover, OBJ implies that the second derivative of the objective function can only have a positive contribution to the diagonal entries of $\nabla_{x_{p}x_{p}}^2 \mathcal{L}_{\mu}$  and hence this matrix is diagonal and nonsingular. 
\end{proof}

\begin{proposition}\label{thm:main-result}
Assume BND, OBJ, LIN, and IND.
Then the Jacobian matrix $\partial F_{\mu}(x,\lambda,\theta)/ \partial (x,\lambda)$ is nonsingular for all feasible $(x,\lambda)$ and all $\theta \in [0,1]$.
\end{proposition}

\begin{proof}
The Jacobian matrix $\partial F_{\mu}(x,\lambda,\theta)/ \partial (x,\lambda)$ has the following block structure:
\[
\frac{\partial F_{\mu}(x,\lambda,\theta)}{\partial (x,\lambda)} =
\begin{pmatrix}
\nabla_{xx}^2 \mathcal{L}_{\mu} & B^T \\
B & 0 \\
\end{pmatrix},
\]
where $\nabla_{xx}^2 \mathcal{L}_{\mu}$ has the block-diagonal structure as in Lemma \ref{lemma:hessian2} and $B = \nabla_x c(x, \theta)$. 

From Lemma~\ref{lem:lin}, we know that there exists a subset of constraints $c^S$ such that $\nabla_{x_r} c^S$ has full-rank. 
Then, through elementary row operations, we can transform the Jacobian such that $\nabla^2_{x_r x_r} \mathcal L_\mu$ becomes a nonsingular, upper-diagonal matrix $D$ and 
\[
\frac{\partial F_{\mu}(x,\lambda,\theta)}{\partial (x,\lambda)} =
\begin{pmatrix}
A & B^T \\
B & 0 \\
\end{pmatrix},
\quad \text{ where } \quad 
A:=
\begin{pmatrix}
\nabla_{x_p x_p}^2 \mathcal{L}_{\mu} & C \\
0 & D
\end{pmatrix}.
\]
By construction, $A$ is a nonsingular, upper-diagonal matrix. 
Taking the Schur complement with respect to the matrix $A$ and since (IND) implies that 
$B$ is full rank, we have
\[
\rank \frac{\partial F_{\mu}(x,\lambda, \theta)}{\partial (x,\lambda)} = \rank A + \rank (B A^{-1} B^T) =\rank A +\rank B = |x| + |c|,
\]
which is exactly the dimensionality of our Jacobian.  Hence, the Jacobian matrix $\partial F_{\mu}(x,\lambda, \theta) / \partial (x,\lambda)$ is nonsingular.
\end{proof}

\begin{theorem}
Assume BND, OBJ, LIN, and IND.  Then the optimization problem $F_{\mu}(x,\lambda,\theta)=0$ is zero-convex and path stable.
\end{theorem}
\begin{proof}
Zero-convexity follows directly by OBJ, LIN and the way the constraints (\ref{eq:homotopic-equation-1}) -- (\ref{eq:alias}) were constructed.  
Path stability follows from Proposition~\ref{thm:main-result}.
\end{proof}

\begin{remark}
In production, we can drop the alias variable $\widetilde{Q}$.  If this were not the case, we could create a bifurcation in the non-aliased problem, and introduce it into the aliased problem, which would be a contradiction.
\end{remark}

\section{Solution algorithm}
\label{section:algorithm}

We now summarize the full solution algorithm \cite{Baayen2018}:

\begin{enumerate}
\item Set $\theta=0$.
\item Store seed solution, if any.
\item Continue until $\theta=1$:
\begin{enumerate}
\item Retrieve stored solution and set as initial guess.
\item Solve and store solution.
\item Increase $\theta$.
\end{enumerate}
\end{enumerate}

Note that the subproblem for the first step of the algorithm has linear constraints.  As such, a feasible seed solution is readily obtained using standard techniques such as linear programming.  This subproblem is also convex, whence every solution is also globally optimal.  The latter condition ensures a well-defined starting point for the homotopy process.

The homotopy parameter $\theta$ is increased with a fixed increment every time.  This increment should be chosen to be sufficiently small for the algorithm to converge.

In the next section, we illustrate the method with an example problem.

\section{Example}

We consider a two-reservoir system with a constant upstream
inflow boundary condition as in Figure \ref{fig:example-topology}. All reservoir outflow is routed though turbines; for clarity of the results, spillways are not considered here (but are readily added). The model parameters and initial conditions are summarized
in Table \ref{tab:example-parameter}.
\begin{figure}[ht]
\centering
\begin{tikzpicture}
\draw[->] (-1, 0.75) node[anchor=west]{$\, \, Q_i$}
  -- (0, 0.25) node{};
  
\draw (0,0.5) node[anchor=north]{}
  -- (0,0) node[anchor=north]{}
  -- (0.5,0) node[anchor=north]{upstream} node[anchor=south]{$H_u$}
  -- (1,0) node[anchor=north]{}
  -- (1,0.5) node[anchor=south]{};

\draw[->] (1, 0.25) node[anchor=west]{$\, \, Q_u$}
  -- (2, -0.25) node{};

\draw (2,0.0) node[anchor=north]{}
  -- (2,-0.5) node[anchor=north]{}
  -- (2.5,-0.5) node[anchor=north]{downstream} node[anchor=south]{$H_d$}
  -- (3,-0.5) node[anchor=north]{}
  -- (3,0.0) node[anchor=south]{};
  
\draw[->] (3, -0.25) node[anchor=west]{$\, \, Q_d$}
  -- (4, -0.75) node{};
  
\draw (3.75, -0.75) node{}
  -- (4.25, -0.75) node[anchor=north]{$H_t$};

\end{tikzpicture}
\caption{Schematic overview of example reservoir system.}
\label{fig:example-topology}
\end{figure}
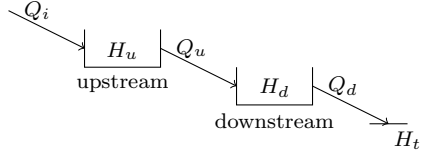

\begin{table}[ht]
\begin{tabular}{l|l|l}
\centering
Parameter & Value & Description \\
\hline
$ N $ & $48$  & Number of time steps \\
$ \Delta t $ & $1$ h & Time step size \\
$ Q_i $ & $100$ m$^3$/s  & Upstream inflow \\
$ \eta_{u,d} $ &  $ 0.85 $ & Turbine efficiency (upstream and downstream) \\
$ g$ &  $ 9.81 $ m/s$^2$ &  Gravitational constant \\
$ \rho $ &  $ 1000 $ kg/m$^3$ & Density of water \\
$ A_{u} $ & $ 10^5 $ m$^2$  & Surface area of the upstream reservoir \\
$ A_{d} $ & $10^5  $ m$^2$ & Surface area of the downstream reservoir\\
$ H_{b, u} $ & $1000$  m & Bottom level of the upstream reservoir \\
$ H_{b, d} $ & $900 $ m & Bottom level of the downstream reservoir \\
$ H_{t} $ & $800$ m & Tailwater level  \\
$ (\Delta H_u)_0 $ & $ 80$  m &  Upstream level difference in linear model \\
$ (\Delta H_d)_0 $ & $ 125$  m  &  Downstream head difference in linear model \\
$ H_{u}(t_{0}) $ & $1005$  m  & Initial level of the upstream reservoir \\
$ H_{d}(t_{0}) $ & $925$  m  & Initial level of the downstream reservoir \\
$ H_{u,max} $ & $ 1030$  m  & Maximum level of the upstream reservoir \\
$ H_{d,max} $ & $930$  m  & Maximum level of the downstream reservoir \\
$ Q_{u,max} $ & $ 100 $  m$^3$/s  & Maximum release of the upstream reservoir \\
$ Q_{d,max} $ & $100$  m$^3$/s  & Maximum release of the downstream reservoir \\
$ P_{u,max} $ & $ 1 $  GW & Maximum generation of the upstream reservoir \\
$ P_{d,max} $ & $1$ GW  & Maximum generation of the downstream reservoir 
\end{tabular}
\caption{Parameters for the example problem.}
\label{tab:example-parameter}
\end{table}

Our optimization objective is to maximize the total hydroelectric generation over the planning horizon:
\[
\max \sum_{j=1}^N P_{u}(t_j) + P_{d}(t_j) 
\]
subject to the additional constraints 
\begin{align*}\label{eq:constraints}
H_{b,u} & \le H_{u} (t_j) \le H_{u,max}, &  H_{b,d}  & \le H_{d} (t_j)  \le H_{d,max},\\
0 & \le Q_{u} (t_j) \le Q_{u,max} , &  0 & \le  Q_{d} (t_j) \le Q_{d,max},\\
0 & \le P_{u} (t_j) \le P_{u,max} , & 0 & \le  P_{d} (t_j) \le P_{d,max}
\end{align*}
for all $t_j$.

For this example, we apply a constant turbine efficiency as well as the linear volume-level relation
\begin{align*}
V_u (t_j) & = A_u \big( H_u (t_j) -H_{b,u} \big), \\
V_d (t_j)  & = A_d \big( H_d (t_j) -H_{b,d} \big).
\end{align*}

\begin{figure}[ht]
\centering
\subfloat[Upstream reservoir level $H_u = H_u (\theta , t)$]{\includegraphics[scale=0.4]{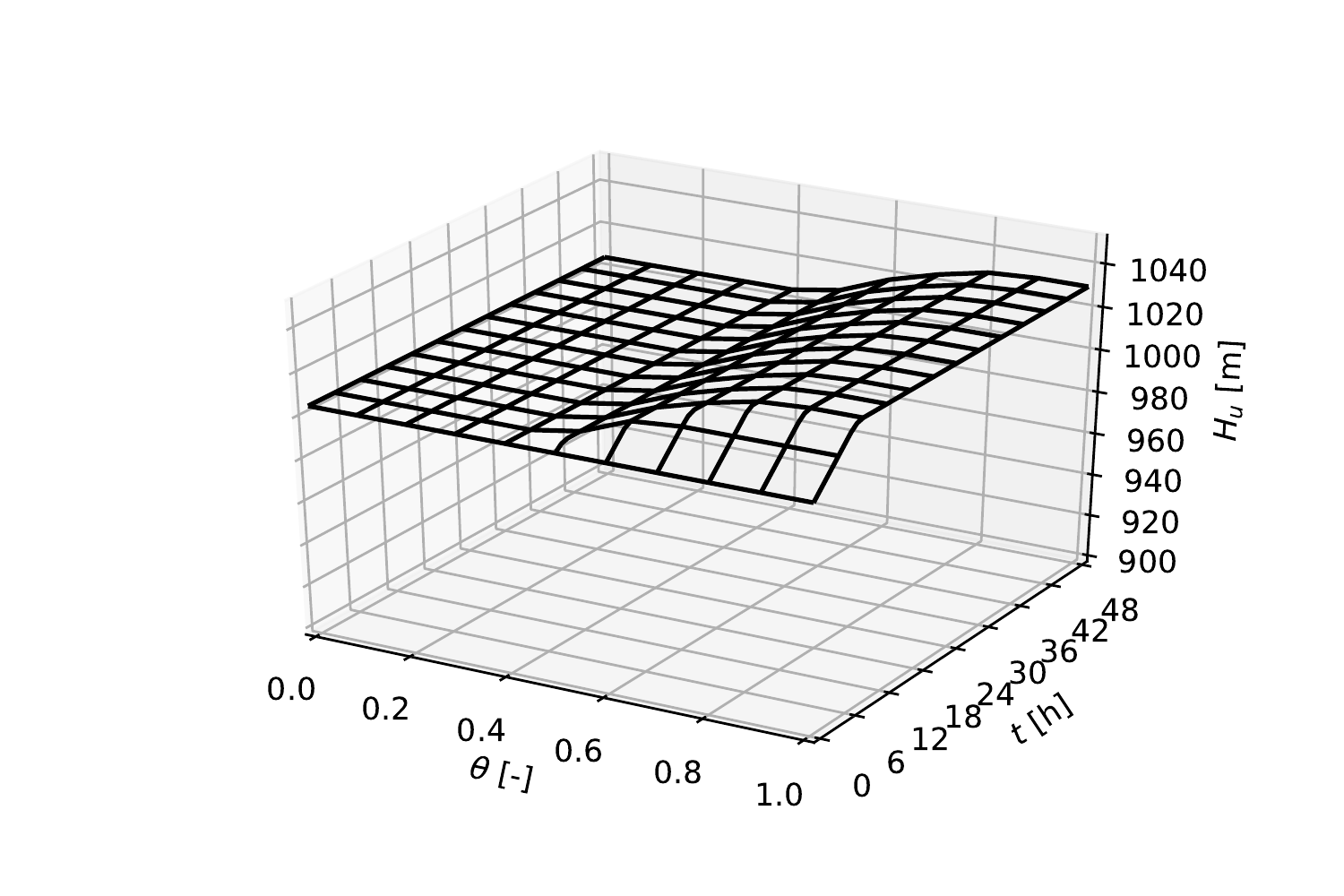}}
\subfloat[Downstream reservoir level $H_d = H_d (\theta , t)$]{\includegraphics[scale=0.4]{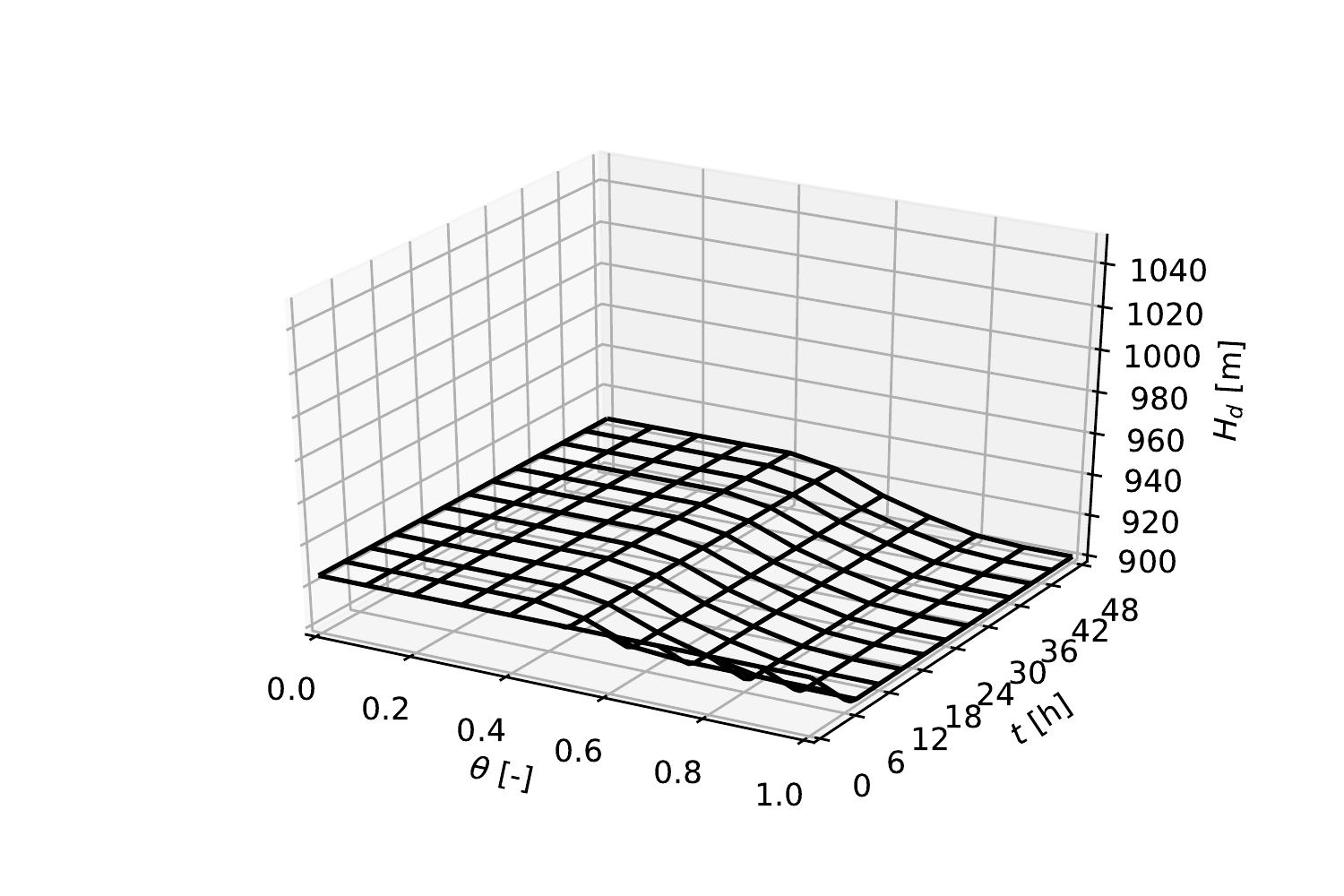}}
\\
\subfloat[Upstream release $Q_u = Q_u (\theta , t)$]{\includegraphics[scale=0.4]{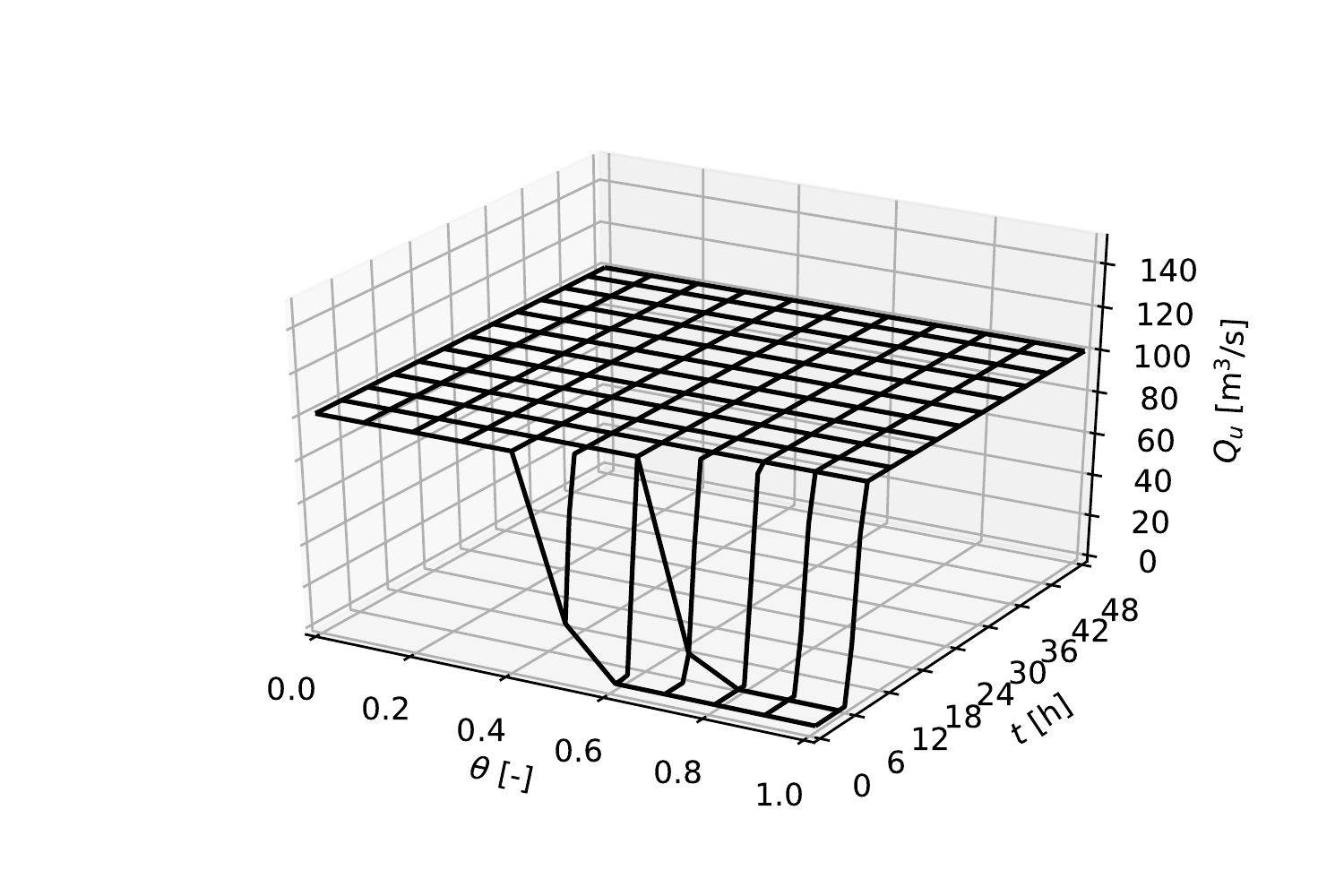}}
\subfloat[Downstream release $Q_d = Q_d (\theta , t)$]{\includegraphics[scale=0.4]{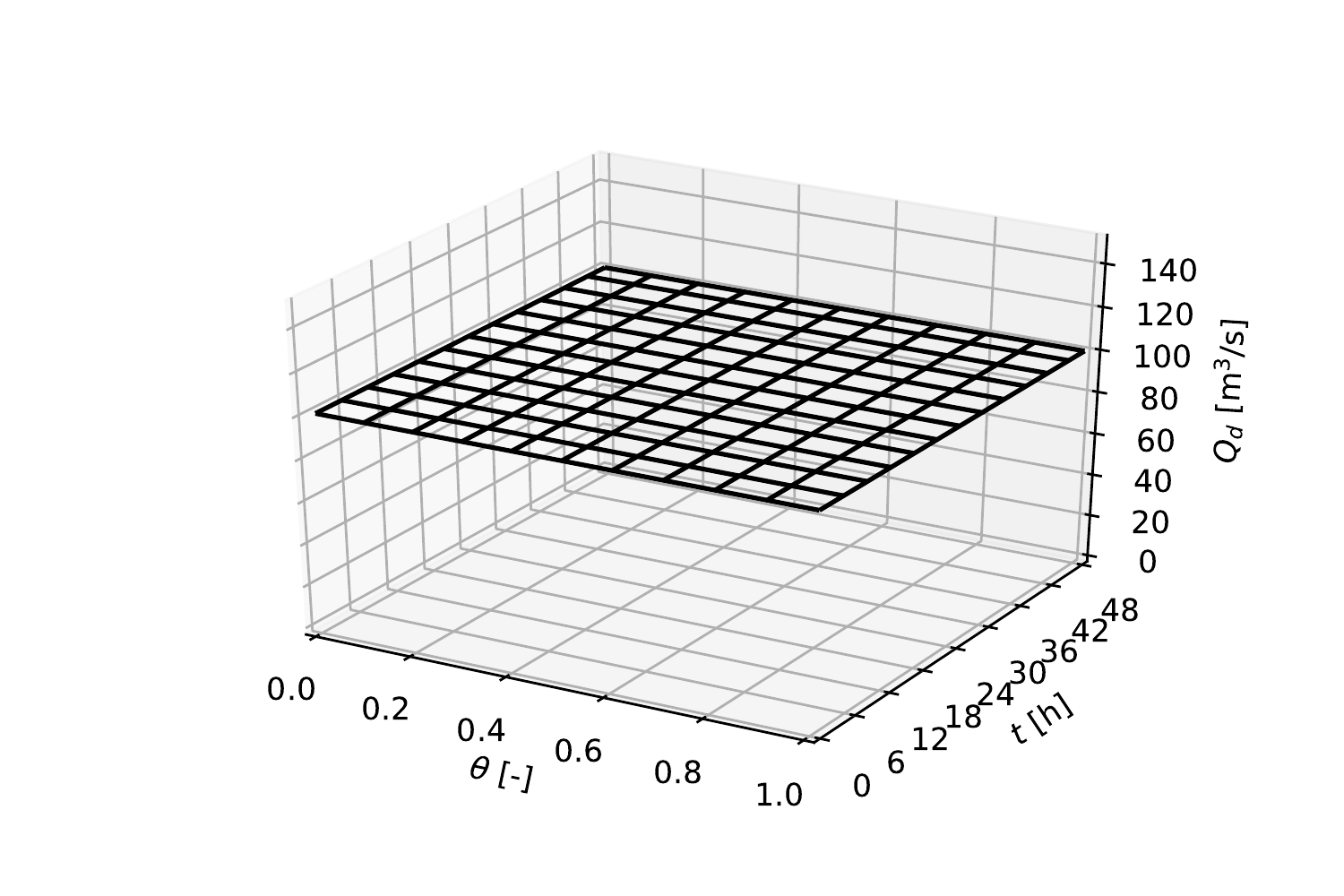}}
\\
\subfloat[Upstream generation $P_u = P_u (\theta , t)$]{\includegraphics[scale=0.4]{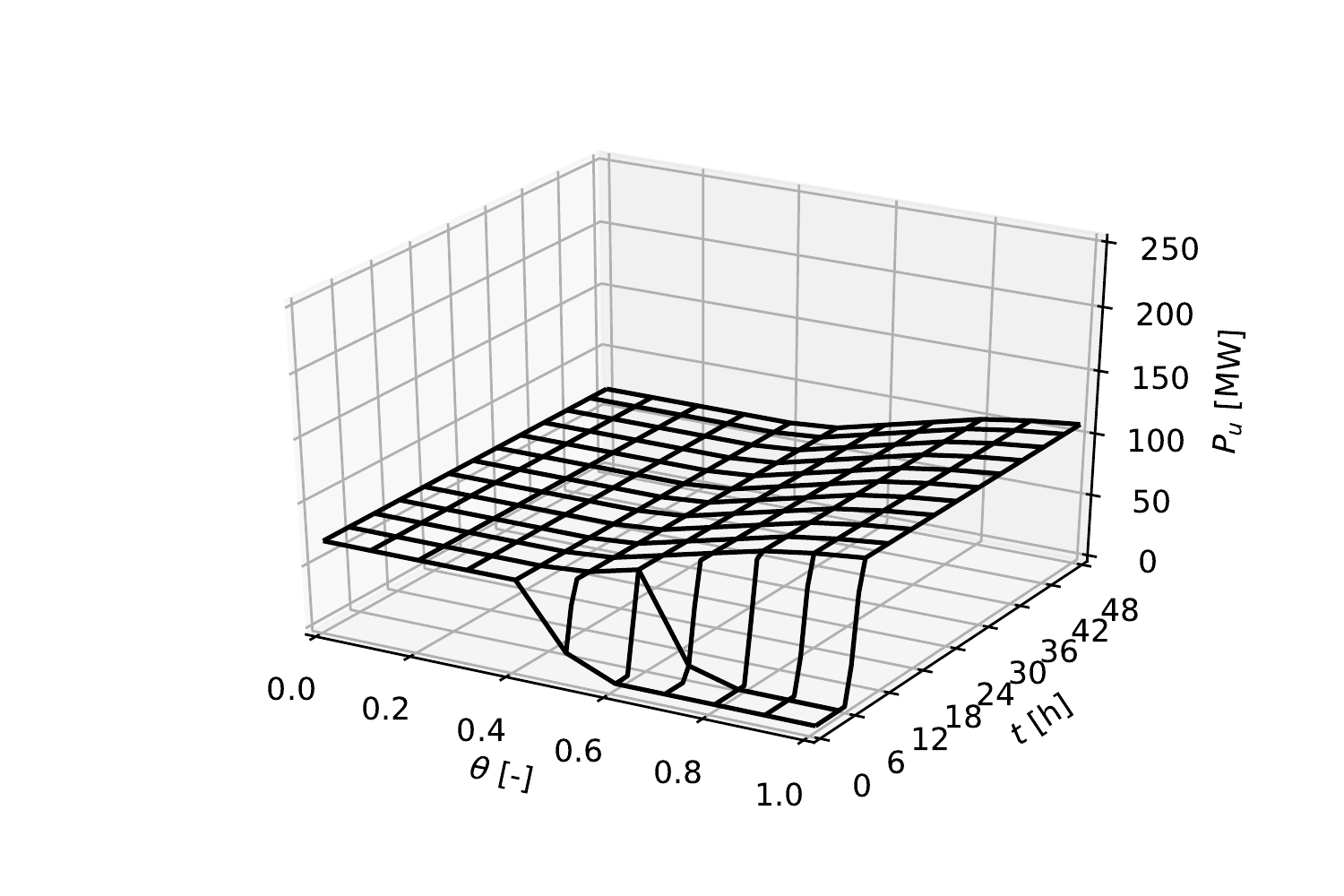}}
\subfloat[Downstream generation $P_d = P_d (\theta , t)$]{\includegraphics[scale=0.4]{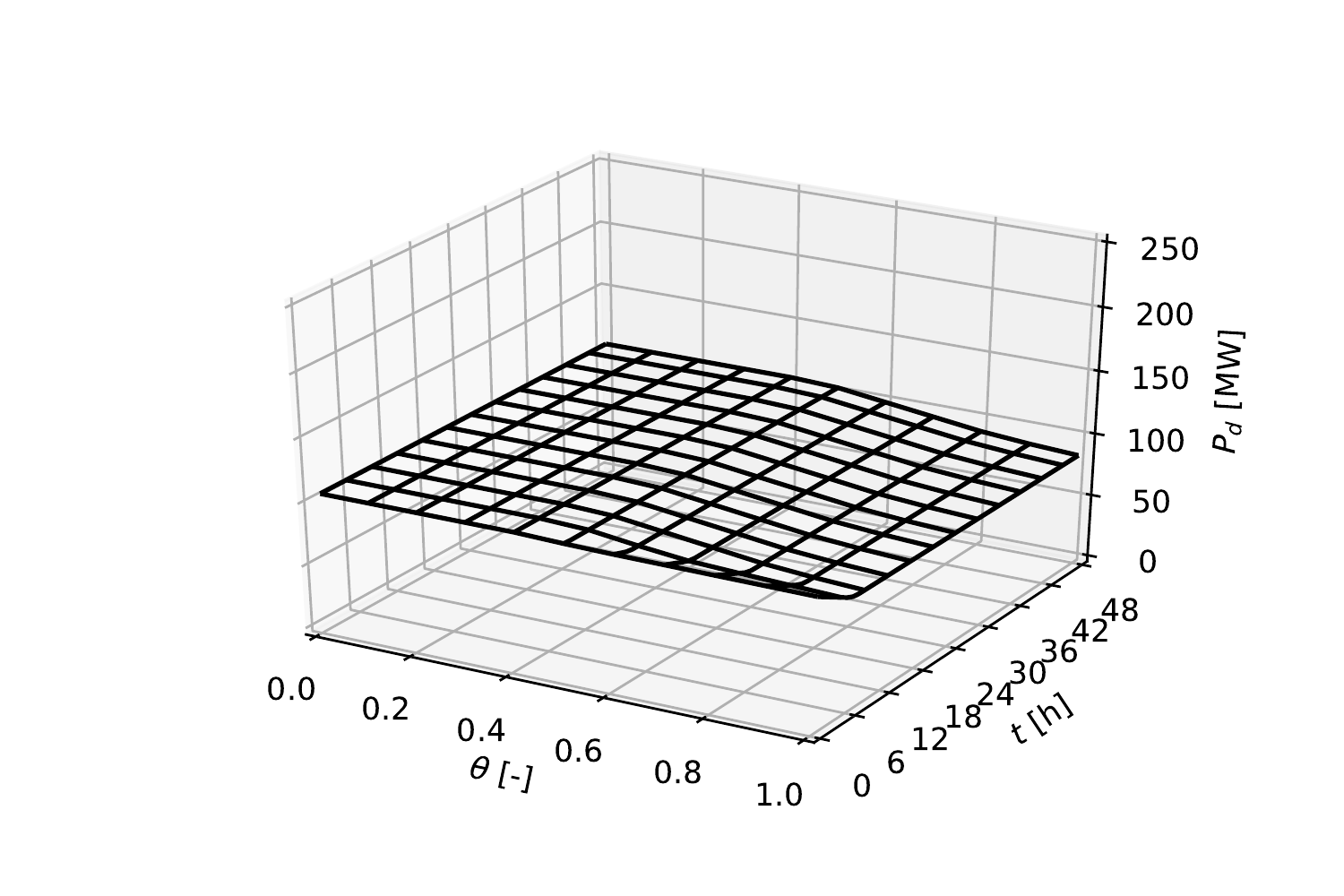}}
\\
\subfloat[Total generation $P = P (\theta , t)$]{\includegraphics[scale=0.4]{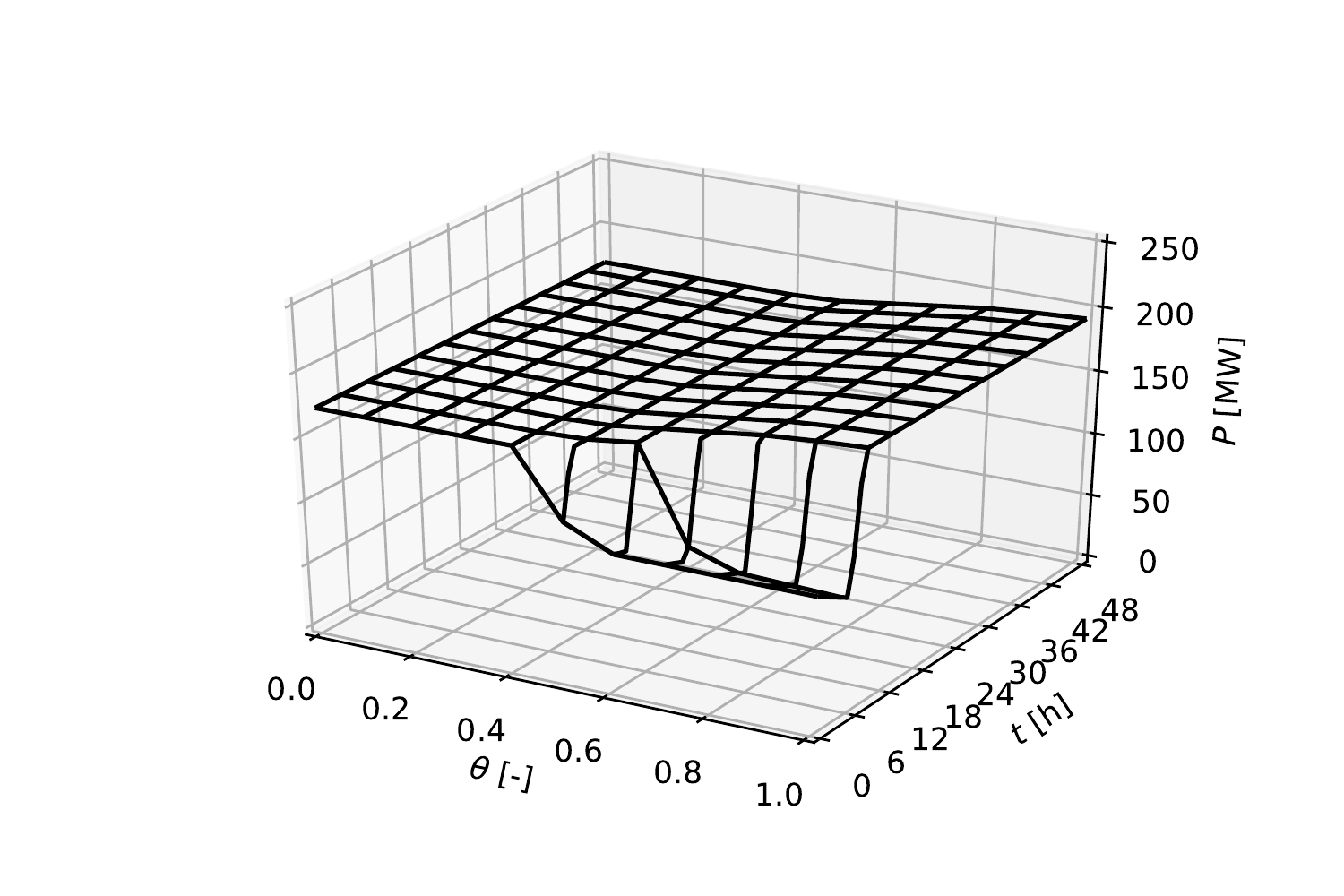}}
\subfloat[Upsteram inflow $Q_i$]{\includegraphics[scale=0.4]{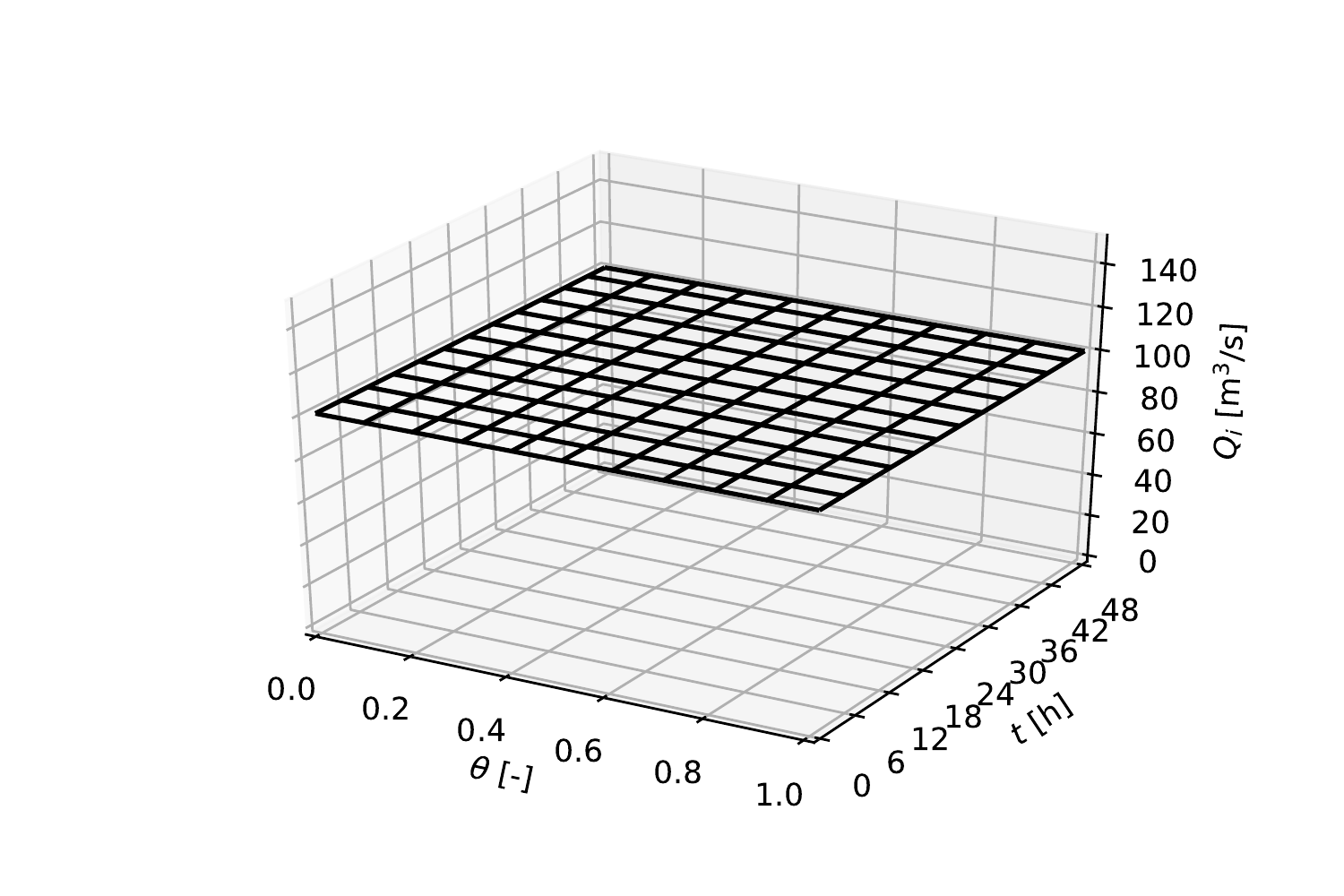}}
\caption{Evolution of hydraulic variables as a function of the homotopy parameter
 $\theta$ and time $t$} \label{fig:theta,time}
\end{figure}

This optimization problem was coded in Python using the CasADi package \cite{Andersson2012} for algorithmic differentiation, and connected to the IPOPT optimization solver \cite{Wachter2006}. The results of applying the method are shown in Figure \ref{fig:theta,time}, where the simultaneous deformation of all hydraulic variables is clearly visible. 

For the linear model, i.e., for $\theta=0$, generation is a function of turbine flow only and hence the optimal solution is to realize the maximum allowed turbine flow at all times.  For the fully non-linear model, i.e., for $\theta=1$, generation is proportional to the product of turbine flow \emph{and} head difference.  In this case, the optimal solution is to \emph{first} realize a higher head difference, and \emph{then} to produce the maximum allowed turbine flow. During the $48$ hours considered in this planning problem, the non-linear solution produces approximately $310$ MWh  more energy than the linear solution (an increase of approximately $4$\%).

On a MacBook Pro with 2.9 GHz Intel Core i5 CPU, the example takes approximately 550 ms to solve.

The complete source code is available online at \url{https://github.com/jbaayen/hydropower-homotopy-example/blob/master/Example.ipynb}.

\section{Conclusions}

In this paper we presented a methodology to carry out numerical optimal control of reservoir systems with hydroelectric generation.  By introducing a suitable parametric deformation, we were able to derive a deterministic algorithm to arrive at a local optimum of the nonconvex optimization problem.  This optimum is consistent in the sense that it is fully determined by the choice of parametric deformation. The approach hinges on the notions of \emph{zero-convexity} and \emph{path stability}, i.e., the property of a parametric optimization problem that I) it has linear constraints at the parameter starting value and II) that when computing a path of solutions as a function of the parameter, no bifurcations arise.  Path stability ensures determinism and numerical stability of the solution process, and renders this class of numerical optimal control problems suitable for deployment in decision support systems based on model predictive control.  Finally, with an example problem, we illustrate how even in simple cases taking head variation into account can result in a $4$\% increase of generation output.

\section{Conflict of Interest}

The authors declare that they have no conflict of interest.

\bibliographystyle{spmpsci}
\bibliography{paper}

\end{document}